\documentclass[psamsfonts]{amsart} 
\usepackage{amssymb}

\usepackage[dvips]{graphicx}

\usepackage{graphicx}
\usepackage{amssymb}                       
\usepackage{amsmath}
\usepackage{color}
\usepackage{times}

\usepackage[unicode,bookmarks,colorlinks]{hyperref}
\hypersetup{
}

\definecolor{mahogany}{cmyk}{0, 0.77, 0.87, 0}
\definecolor{salmon}{cmyk}{0, 0.53, 0.38, 0}
\definecolor{melon}{cmyk}{0, 0.46, 0.50, 0}
\definecolor{yellowgreen}{cmyk}{0.44, 0, 0.74, 0}
\definecolor{brickred}{cmyk}{0, 0.89, 0.94, 0.28}
\definecolor{OliveGreen}{cmyk}{0.64, 0, 0.95, 0.40}
\definecolor{RawSienna}{cmyk}{0, 0.72, 1.0, 0.45}
\definecolor{ZurichRed}{rgb}{1, 0, 0} 

\usepackage{fancyhdr}
\usepackage{amsmath,amstext,amssymb,amsopn,amsthm}
\usepackage{amsmath,amssymb,amsthm}
\usepackage[mathscr]{eucal}

\begin{document}

\newtheorem{lemma}[thm]{Lemma}
\newtheorem{corr}[thm]{Corollary}
\newtheorem{proposition}{Proposition}
\newtheorem{theorem}{Theorem}[section]
\newtheorem{deff}[thm]{Definition}
\newtheorem{case}{Case}
\newtheorem{prop}[thm]{Proposition}
\newtheorem{example}{Example}

\newtheorem{corollary}{Corollary}

\theoremstyle{definition}
\newtheorem{remark}{Remark}

\numberwithin{equation}{section}
\numberwithin{definition}{section}
\numberwithin{corollary}{section}

\numberwithin{theorem}{section}

\numberwithin{remark}{section}
\numberwithin{example}{section}
\numberwithin{proposition}{section}

\newcommand{\calD}{\bD}

\newcommand{\conjugate}[1]{\overline{#1}}
\newcommand{\abs}[1]{\left| #1 \right|}
\newcommand{\cl}[1]{\overline{#1}}
\newcommand{\expr}[1]{\left( #1 \right)}
\newcommand{\set}[1]{\left\{ #1 \right\}}

\newcommand{\calC}{\mathcal{C}}
\newcommand{\calK}{\mathcal{K}}
\newcommand{\calS}{\mathcal{S}}
\newcommand{\calE}{\mathcal{E}}
\newcommand{\calF}{\mathcal{F}}
\newcommand{\Rd}{\mathbb{R}^d}
\newcommand{\BR}{\bD(\Rd)}
\newcommand{\R}{\mathbb{R}}
\newcommand{\al}{\alpha}
\newcommand{\RR}[1]{\mathbb{#1}}
\newcommand{\bR}{\mathrm{I\! R\!}}
\newcommand{\ga}{\gamma}
\newcommand{\om}{\omega}
\newcommand{\A}{\mathbb{A}}
\newcommand{\bH}{\mathbb{H}}

\newcommand{\bb}[1]{\mathbb{#1}}
\newcommand{\bI}{\bb{I}}
\newcommand{\bN}{\bb{N}}

\newcommand{\uS}{\mathbb{S}}
\newcommand{\M}{{\mathcal{M}}}
\newcommand{\calB}{{\mathcal{B}}}

\newcommand{\W}{{\mathcal{W}}}

\newcommand{\m}{{\mathcal{m}}}

\newcommand {\mac}[1] { \mathbb{#1} }

\newcommand{\bD}{\mathbb{D}}

\newcommand{\bC}{\Bbb C}

\newtheorem{rem}[theorem]{Remark}
\newtheorem{dfn}[theorem]{Definition}
\theoremstyle{definition}
\newtheorem{ex}[theorem]{Example}
\numberwithin{equation}{section}

\newcommand{\Pro}{\mathbb{P}}
\newcommand\F{\mathcal{F}}
\newcommand\E{\mathbb{E}}
\newcommand\e{\varepsilon}
\def\H{\mathcal{H}}
\def\t{\tau}

\newcommand{\blankbox}[2]{%
  \parbox{\columnwidth}{\centering
    \setlength{\fboxsep}{0pt}%
    \fbox{\raisebox{0pt}[#2]{\hspace{#1}}}%
  }%
}

\title[CZ Operators and Martingale Transforms]{On a class of Calder\'on-Zygmund Operators Arising from Projections of Martingale Transforms}

\author{Michael Perlmutter*}\thanks{* Supported in part  by NSF Grant
\# 0603701-DMS under PI Rodrigo Ba\~nuelos}
\address{Department of Mathematics, Purdue University, West Lafayette, IN 47907, USA}
\email{mperlmut@math.purdue.edu}

\begin{abstract}   We prove that a large class of operators, which arise as the projections of martingale transforms of stochastic integrals with respect to Brownian motion, as well as other closely related operators, are in fact Calder\'on--Zygmund operators.  Consequently, such operators are not only bounded on $L^p$, $1<p<\infty$, but also satisfy  weak-type inequalities.  Unlike the boundedness on $L^p,$ which can be obtained directly from the Burkholder martingale transform inequalities, the weak-type estimates do not follow from the corresponding martingale results.   
\end{abstract}

%

\maketitle


\section{Introduction and statement of results}
Martingale inequality methods provide a powerful tool to study the $L^p$ boundedness, $1<p<\infty$,  of the basic Calder\'on-Zygmund singular integral operators on $\R^n$.  An advantage of these techniques is that they give very good information on the size of these $L^p$ bounds and, in particular,  provide constants that are independent of the dimension.  These same arguments can be used to extend results from  $\R^n$ to manifolds and to the Ornstein-Uhleneck case.  For some applications of these methods we refer the reader to \cite{Ban2}, \cite{BanBau}, \cite{BanMen}, \cite{BanWan}, \cite{BanOse}, \cite{BorJanVol}, \cite{GeiMonSmi}, \cite{NazVol}, \cite{Ose},  and the many references provided there.  However, as powerful as these techniques are, weak-type martingale inequalities cannot be directly transferred to singular integral operators.   For example, while Burkholder's celebrated $L^p$ inequalities, $1<p<\infty$,  for martingale transforms \cite{Bur2}, with his famous bound $``(p^*-1)",$ gives the same $L^p$ bound for many singular integral operators,  his weak-type martingale transform bound $``2"$ provides no information for the weak-type inequalities  of  those operators.  This is due to the fact that the probabilistic representation of such operators involves the use of conditional  expectation which does not preserve weak-type inequalities. The purpose of this paper is to show that a very general class of operators, including many of the operators considered in \cite{BanMen} and \cite{BanWan}, which arise as the projections of martingale transforms, are Calder\'on-Zygmund operators. Such operators  do not have to be of convolution type.
Once we know that these are Calder\'on-Zygmund operators, they then satisfy all the properties of such operators, including their weak-type boundedness. This does not, of course, answer  important questions that have been of interest to many people, starting with Stein \cite{Ste1} in the case of the Riesz transforms:  do these operators have weak-type bounds independent of the dimension?  Do weak-type inequalities hold for Riesz transforms on Wiener space?  For a more precise formulation of these questions, see \cite{Ban1}. 

For the rest of this paper, and following standard terminology (see for example \cite[p.175]{Gra}),  we will say that an operator $T$ acting on the  Schwartz space of rapidly decreasing functions on $\R^{n}$ is a Calder\'on-Zygmund (CZ) operator if it admits a bounded extension to $L^2(\mathbb{R}^n)$ and is of the form 
\begin{equation}
\label{CalZig} Tf(x)= \int_{\mathbb{R}^n} K(x, \tilde{x}) f(\tilde{x})d\tilde{x} 
\end{equation}
where the kernel $K$ satisfies  the following conditions
\begin{align}
\label{size}|K(x,\tilde{x})| &\leq \frac{\kappa}{|x-\tilde{x}|^{n}}\\
\label{smoothness1}|\nabla_x K(x,\tilde{x})| &\leq \frac{\kappa}{|x-\tilde{x}|^{n+1}} \\
\label{smoothness2}|\nabla_{\tilde{x}} K(x,\tilde{x})| &\leq \frac{\kappa}{|x-\tilde{x}|^{n+1}},
\end{align}
for some universal constant $\kappa$ whenever $x\neq \tilde{x}$. (In (\ref{CalZig}) the integral is defined in the principal-value sense if $K(x,\tilde{x})$ is singular along the diagonal $\{x=\tilde{x}\}$. This convention will be used throughout the entirety of this paper.) If there exists a function $\bar{K}$, defined on $\mathbb{R}^n\setminus\{0\}$, so that $\bar{K}(x-\tilde{x})= K(x,\tilde{x})$ for all $x\neq \tilde{x}$, then we say that $T$ is of convolution type. The Hilbert, Riesz, and Beurling-Ahlfors transforms discussed below are basic examples of CZ operators of convolution type which give rise to interesting Fourier multipliers.  It is well known (see for example \cite[p.183]{Gra}) that CZ operators are strong-type $(p, p)$ for $1<p<\infty$ and are weak-type $(1, 1)$.  More precisely, there exists universal constants $C_{p,n,\kappa}$, depending only on $p,n,$ and $\kappa$, such that
\begin{equation}
\label{strongtype} \|Tf\|_p \leq C_{p,n,
\kappa} \|f\|_p,\,\,\,\,\, 1<p<\infty 
\end{equation} 
and 
\begin{equation}
\label{weaktype} |\{x: |Tf(x)| > \lambda\}| \leq \frac{C_{1,n,\kappa}}{\lambda} \|f\|_1.
\end{equation}
    The purpose of this paper is to prove the following theorem.  As we shall see, in the case that $\alpha=1$ or $2$, these operators are the conditional expectations of martingale transforms which were used in \cite{BanWan} and \cite{BanMen} respectively. Background information on $\alpha$-stable processes will be provided at the beginning of the next section.

\begin{theorem}\label{stable} Let $0 < \alpha \leq 2$. Let $(X_t)_{t> 0}$ be a rotationally-invariant (symmetric) $\alpha$-stable process on $\mathbb{R}^n$ and let. 
$\varphi$ denote the density of $X_1$. For $y\geq 0$, let $\varphi_y(x)= \frac{1}{y^n}\varphi(\frac{x}{y})$.
Let  $A(x,y)=(a^{i,j}(x,y))$ be an $(n+1)\times(n+1)$ matrix-valued function with 
\begin{equation}
 \|A\|=\|\sup_{|v|\leq 1}(|A(x,y)v|)\|_{L^\infty(\mathbb{R}^n\times [0,\infty))} <\infty.
\end{equation} 
Further assume that $a^{i,j}(x,y)=a^{i,j}(y)$ is independent of $x$ whenever $i$ or $j=n+1$. 
 Consider the kernel 
\begin{equation}
\label{StableK} K_A(x,\tilde{x}) = \int_0^\infty \int_{\mathbb{R}^n} 2yA(\bar{x},y)\nabla \varphi_y(\bar{x}-\tilde{x})\nabla \varphi_y(\bar{x}-x) d\bar{x}dy,
\end{equation}
where $\nabla = (\partial_{x_1},\ldots,\partial_{x_n},\partial_y)$.
 Then the operator 
\begin{equation*}
T_{A}f(x) = \int_{\mathbb{R}^n} K(x,\tilde{x}) f(\tilde{x})d\tilde{x}\
\end{equation*}
 is a CZ operator. 
\end{theorem}
\begin{remark}\label{densities}
If we make the additional assumption that $a^{i,j}(y)=0$ whenever $i$ or $j=n+1$, we may also write our kernel in terms of the density of $X_t$, which we denote $\psi_t$. It is well known (see e.g. \cite{Ber}) that $\psi_t$ obeys the scaling relation $\psi_t(x) = \frac{1}{t^{n/ \alpha}}\psi(\frac{x}{t^{1/ \alpha}})$ which implies $\varphi_{t^{1/ \alpha}}=\psi_t$. Therefore, after a simple change of variables we see that 
\begin{equation}
K_A(x,\tilde{x}) = \int_0^\infty \int_{\mathbb{R}^n} \frac{2}{\alpha}t^{\frac{2}{\alpha}-1}A(\bar{x},t^{1/\alpha})\nabla \psi_t(\bar{x}-\tilde{x})\nabla \psi_t(\bar{x}-x) d\bar{x}dt.
\end{equation}
The reason why we need the assumption that  $a^{i,j}(y)=0$ whenever $i$ or $j=n+1$ is because these entries correspond to ``vertical'' derivatives with respect to the dilation parameter $t$, and the change of variables $y=t^{1/\alpha}$ does not commute with the taking of vertical derivatives.  
\end{remark}

The rest of this paper is organized as follows. In section two, we give background information on CZ operators, martingale transforms, and the connection between the two topics as well as a brief introduction to L\'evy processes. In section three, we give the proof of theorem \ref{stable}, and in section four, we give our closing remarks and discuss additional properties of CZ operators.   

\section{Preliminaries}

We start with a brief introduction to an important class of L\'evy processes called rotationally invariant $\alpha-$stable processes. Recall that a L\'evy process on $\mathbb{R}^n$ is a stochastically continuous process, $(X_t)_{t\geq 0}$, with stationary and independent increments such that $X_0=0$ a.s. By stochastic continuity we mean that for every $\epsilon > 0$,
\begin{equation*}
\lim_{t\searrow 0} \mathbb{P}(|X_t| > \epsilon ) =0.
\end{equation*}
The celebrated L\'evy-Khintchine formula says that if $X_t$ a L\'evy process, its characteristic function is given by $\mathbb{E}(e^{i\xi\cdot X_t}) = e^{t\rho(\xi)}$ where 
\begin{equation*}
\rho(\xi) = ib\cdot \xi - \frac{1}{2} B\xi \cdot \xi + \int_{\mathbb{R}^n} (e^{i \xi\cdot y} - 1 - i(\xi\cdot y)\mathbb{I}_{(|y|\leq 1)}) d\nu (y)
\end{equation*}
with $b\in \mathbb{R}^n$, $B$ a symmetric non-negative $n\times n$ matrix, and $\nu$ a measure satisfying $\nu(\{0\})=0$ and 
\begin{equation*}
\int_{\mathbb{R}^n} \frac{|y|^2}{|y|^2+1} d\nu(y)<\infty.
\end{equation*}

For $0<\alpha\leq 2$, $\rho(\xi) = -|\xi|^\alpha$ gives the rotationally invariant $\alpha-$stable processes.
In the case that $\alpha=2$, $(X_t)_{t\geq 0}$ is Brownian motion (running at twice the usual speed), and the density of $X_t$ is given by 
\begin{equation}
\label{BM}\frac{1}{(4\pi t)^{n/2}}e^{-|x|^2/4t}.
\end{equation}
 If $\alpha=1$, then $(X_t)_{t\geq 0}$ is the Cauchy process and the density of $X_t$ is given by 
\begin{equation}
\label{CP}\frac{\Gamma(\frac{n+1}{2})}{\pi^{(n+1)/2}}\frac{t}{(|x|^2+t^2)^{(n+1)/2}}.
\end{equation}  
 For further background on L\'evy processes we refer the reader to \cite{Ber}, \cite{Blu}, and \cite{Sato}.

We now consider the basic examples of CZ operators which arise as projections of martingale transforms, the Riesz and Beurling-Ahlfors transforms.   For $f\in L^p(\mathbb{R}^n)$, we define the Riesz transform in direction $j$, $1\leq j\leq n,$ by
\begin{equation*}
R_jf(x)=  \frac{\Gamma(\frac{n+1}{2})}{\pi^{(n+1)/2}} \int_{\mathbb{R}^n}\frac{x_j-\tilde{x}_j}{|x-\tilde{x}|^{n+1}}f(\tilde{x})d\tilde{x}.
\end{equation*}
When $n=1$, the Riesz transform is called the Hilbert transform.  Likewise, for $f\in L^p(\mathbb{C})$, we define the Beurling-Ahlfors operator by
\begin{equation*}
Bf(z)=-\frac{1}{\pi} \int_{\mathbb{C}} \frac{f(w)}{(z-w)^2}dw. 
\end{equation*}
(As in (\ref{CalZig}), the above integrals are defined in the principal-value sense.)
These operators are Fourier multipliers with
\begin{equation*}
\widehat{R_jf}(\xi) = m_j(\xi)\widehat{f}(\xi) \quad\text{and}\quad \widehat{Bf}(\xi) = m_B(\xi)\widehat{f}(\xi),
\end{equation*}
where $m_j(\xi) = \frac{i\xi_j}{|\xi|}$ and $m_B(\xi) = \frac{\xi_1^2-\xi_2^2 - 2i\xi_1\xi_2}{|\xi|^2}.$ Therefore, we can decompose the Beurling-Ahlfors transform as
\begin{equation}
\label{decomp} B = R^2_2 - R^2_1 + 2iR_1R_2.
\end{equation}

In \cite{Ste1} Stein showed that for the Riesz transform, the constant in (\ref{strongtype}) can be taken to be independent of the dimension, $n$. 
Gundy and Varopoulos showed in \cite{GunVar} that the Riesz transforms could be interpreted probabilistically as projections of martingale transforms, and from this it again follows that the constant may be taken to be independent of dimension. See  \cite{Ban1} for more on this topic.  These techniques were further explored by Ba\~nuelos and Wang in \cite{BanWan} to prove the sharp inequalities
\begin{equation*}
\| R_jf\|_p \leq C_p\|f\|_p \quad\text{and}\quad \| ((R_jf)^2+f^2)^{1/2}\|_p \leq \sqrt{C_p^2+1} \|f\|_p,
\end{equation*}
where 
\begin{equation*}
p^* = \max\left\{p,\frac{p}{p-1}\right\},\,\,\,\,  \text{and}\,\,\,\, C_p = \cot\left(\frac{\pi}{2p^*}\right).
\end{equation*}
 The first inequality had been proved earlier by Iwaniec and Martin in \cite{IwaMar} using the method of rotations. 

In \cite{Leh}  Lehto showed that the best possible constant in (\ref{strongtype}) for Beurling-Ahlfors transform is at least $(p^*-1)$. Iwaniec conjectured in \cite{Iwa} that it is exactly $(p^*-1)$. In \cite{BanWan} it was shown that the best possible constant for the Beurling-Ahlfors transform is at most $4(p^*-1)$. This constant was reduced to $2(p^*-1)$ by Nazarov and Volberg in \cite{NazVol} using a Littlewood-Paley inequality proved using Bellman functions techniques. The Bellman function in \cite{NazVol} is itself constructed from Burkholder martingale inequalities. In \cite{BanMen} the martingale techniques from \cite{BanWan} are applied to space-time Brownian motion to reproduce the bound $2(p^*-1)$. These martingale  methods were refined in \cite{BanJan} to reduce this constant to $1.575(p^*-1)$,  which is the best known bound as of now valid for all $1<p<\infty$.  We do point out that for $1000<p<\infty$, this bound was improved to $1.4(p^*-1)$ in \cite{BorJanVol}. 
 
With the exception of \cite{NazVol}, the basic idea for the above results is to embed $L^p(\mathbb{R}^n)$ into $\mathbb{M}^p$, the space of p-integrable martingales,  apply a martingale transform, and project back onto $L^p(\mathbb{R}^n)$. This ``factorization" of the operators ``lifts" all the analysis to the martingale setting. We will now give a brief description of this process starting with some background information on martingale transforms and their bounds.  

Let $(X_t)_{t\geq 0}$ be a martingale adapted to the Brownian filtration. Then, we may find an $\mathbb{R}^n$-valued predictable process $H_s$ such that 
\begin{equation*} 
X_t = X_0 + \int_0^t H_s \cdot dB_s,
\end{equation*}
where $B_s$ is a standard Brownian motion.
For an $n \times n$ matrix-valued function, $A(x)$, we define a new martingale 
\begin{equation*}
(A* X)_t =  \int_0^t A(X_s)H_s \cdot dB_s,
\end{equation*}
which we call the martingale transform of $X$ by $A$. In \cite{BanWan} Ba\~nuelos and Wang applied the Burkholder inequalities \cite{Bur2,Bur3} to show
\begin{equation*}\label{Burk1}
\|A*X\|_p \leq \|A\| (p^*-1) \|X\|_p,\,\,\,\,\, 1<p<\infty,
\end{equation*}
where $\|A\| = \|\sup_{|v|\leq 1}|A(x)v|\|_{L^\infty(\mathbb{R}^n)}$ and the norm of a martingale $M_t$ is defined by$\|M\|_p = \sup_{t\geq 0} \|M_t\|_p$. This bound is sharp. 

We now consider how to embed $L^p(\mathbb{R}^n)$ into the space of martingales and project back to $L^p(\mathbb{R}^n)$ using the method developed by Gundy and Varopoulos in \cite{GunVar} and used by Ba\~nuelos and Wang in \cite{BanWan}. Let 
\begin{equation}
p_y(x) = \frac{\Gamma(\frac{n+1}{2})}{\pi^{(n+1)/2}}\frac{y}{(|x|^2+y^2)^{(n+1)/2}}
\end{equation}
be the Poisson kernel for the upper half-space, $\mathbb{R}^{n+1}_+,$ and for $f(x)\in C^\infty_0(\mathbb{R}^n)$, let $(p_y*f)(x) = u_f(x,y)$.  (Note that by (\ref{CP}) $p_y$ is the density of the Cauchy process at time $y$.) Background radiation is a ``time-reversed Brownian motion,'' $(B_t)_{t\leq 0},$ taking values in $\mathbb{R}^{n+1}_+$ such that $B_{-\infty}$ has distribution given by the Lebesgue measure on $\mathbb{R}^n \times \{\infty\}$, and $B_0$ is distributed by the Lebesgue measure on $\mathbb{R}^n\times \{0\}$. We write $B_t=(X_t,Y_t)$ with $X_t$ taking values in $\mathbb{R}^n$ and $Y_t >0$.  The standard rules of stochastic calculus, in particular Ito's formula, hold for the background radiation process. Therefore, $u_f(X_t,Y_t)$ is a martingale and 

\begin{equation*}
f(X_0) = u_f(B_0) =   \int_{-\infty}^0 \nabla u_f(X_s,Y_s) \cdot dB_s,
\end{equation*} 
where $\nabla = (\partial_{x_1},\ldots,\partial_{x_n},\partial_y).$ This allows us to define the martingale transform of $f$ by an $(n+1) \times (n+1)$ matrix-valued function, $A(x,y)$, as 

\begin{equation}
\label{A*f} (A*f) = \int_{-\infty}^0 A(X_s,Y_s) \nabla u_f(X_s,Y_s) \cdot dB_s.
\end{equation}

The random variable $A*f$ is not a function of the endpoint, $X_0$. This motivates us to define a projection operator by averaging the integral in  $(\ref{A*f})$ over all paths ending at $x$, that is,

\begin{equation*}
T_Af(x) = E\left(\int_{-\infty}^0 A(X_s,Y_s) \nabla u_f(X_s,Y_s) \cdot dB_s|X_0 = x\right).
\end{equation*}
 It is known (see \cite{Ban1}) that $E(|(f(B_0)|^p) = \int_{\mathbb{R}^n}|f(x)|^pdx$, which implies  $$\sup_{t\geq 0} \|u_f(B_t)\|_p  = \|f\|_p$$ since $|u_f(B_t)|^p$ is a submartingale. In other words, lifting $f\in L^p(\mathbb{R}^n)$ to the space of martingales does not change its norm. Combining this with the fact that conditional expectation is a contraction in $L^p(\mathbb{R}^n)$, we see that the operator norm of $T_A$ is the same as the operator norm of the martingale transform $X \rightarrow A*X$. Thus, we have 

\begin{equation*}
\|T_Af(x)\|_p \leq (p^*-1)\|A\|\|f\|_p.
\end{equation*}

 It is known (see for example \cite{Bur2}) that martingale transforms are weak-type $(1, 1)$ and in fact  we have the sharp inequality
 $$
 P\{|A*X| >\lambda\}\leq \frac{2\|A\|}{\lambda}\|X\|_1.
 $$ 
  Unfortunately, this does not give us information about the weak-type behavior of $T_A$ because weak-type inequalities are not preserved under conditional expectation. Therefore, we seek to represent $T_A$ as a purely analytic operator by finding a kernel $K_A(x,\tilde{x})$ such that 
  $$T_Af(x) = \int_{\mathbb{R}^n} K_A(x,\tilde{x})f(\tilde{x}) d\tilde{x}.$$ 
  
Let $f,g \in C^\infty_0(\mathbb{R}^n)$ and note that 
\begin{equation*}
g(B_0) = \int_{-\infty}^0 \nabla U_g(B_s) \cdot dB_s
\end{equation*}
by Ito's formula.
Therefore, using basic facts about the covariation of stochastic integrals and the occupation time formula for the background radiation process, (see \cite[p.31 and 57]{Dur} and \cite{GunVar})

\begin{align}
\int_{\mathbb{R}^n} T_Af(x)g(x) dx &= \int_{\mathbb{R}^n}  E\left( \int_{-\infty}^0 A(X_s,Y_s) \nabla u_f(X_s,Y_s) \cdot dB_s|X_0 = x\right)g(x) dx\nonumber\\
&= E\left( \int_{-\infty}^0 A(X_s,Y_s) \nabla u_f(X_s,Y_s) \cdot dB_s g(B_0)\right)\nonumber\\
&= E\left(\int_{-\infty}^0 A(X_s,Y_s) \nabla u_f(X_s,Y_s) \cdot dB_s \int_{-\infty}^0 \nabla u_g(B_s) \cdot dB_s\right)\nonumber\\
&= E\left(\int_{-\infty}^0 A(X_s,Y_s) \nabla u_f(X_s,Y_s) \cdot  \nabla u_g(B_s)  ds\right)\nonumber\\
\label{IPrep}&= \int_0^\infty\int_{\mathbb{R}^n} 2y  A(x,y) \nabla u_f(x,y) \cdot  \nabla u_g(x,y) dxdy.
\end{align}
Using the fact that $\nabla u_f(x,y) = ((\nabla p_y) * f)(x)$ and applying Fubini's theorem, we see that we have
\begin{equation*}
\label{PoissonK} K_A(x,\tilde{x}) = \int_0^\infty \int_{\mathbb{R}^n} 2yA(\bar{x},y)\nabla p_y(\bar{x}-\tilde{x})\nabla p_y(\bar{x}-x) d\bar{x}dy.
\end{equation*} 
Note that this is the kernel from (\ref{StableK}) with $\alpha=1.$

If we define $A_j = (a^j_{l,m})$ by 
 \begin{displaymath}
   a^j_{l,m} = \left\{
     \begin{array}{lr}
       1 & l=n+1,\:m=j\\
       -1 &  l=j, \: m=n+1\\
0&\text{otherwise}
     \end{array}
   \right\},
\end{displaymath} 
then plugging into (\ref{IPrep}) and Fourier transforming shows that $T_{A_j} = R_j$. In fact, if $A$ is any matrix with constant coefficients, $T_A$ will be a linear combination of first and second order Riesz transforms and the identity. Moreover, if $A(x,y) = A(y)$ is independent of $x$ and $\|A\|<\infty$, then $T_A$ is a Fourier multiplier.
 
The approach of \cite{BanMen} is similar, but uses space-time Brownian motion and the heat kernel for the one-half Laplacian,
 \begin{equation}
h_t(x) = \frac{1}{(2\pi t)^{n/2}}e^{-|x|^2/2t},
\end{equation}
instead of background radiation and the Poisson kernel. (We remark that $h_t$  is the density of a standard Brownian motion at time $t$. Observe that this is, up to a simple time change, $t=2s$, the density of the stable process given in (\ref{BM}).) Fix $T>0$, and let $Z_t = (B_t,T-t)$ for $0<t<T$ where $B_t$ is Brownian motion on $\mathbb{R}^n$ with initial distribution given by the Lebesgue measure. We now let $u_f(x,t)$ denote the extension of $f$ to the upper half-space by convolution with $h_t$. Ito's formula shows that $u_f(Z_t)$ is a martingale. 
For an $n\times n$ matrix-valued function, $A(x,t)$, we define a martingale transform and a projection operator by 
\begin{align*}
A \ast f &= \int_{0}^T A(B_s,T-s)\nabla_x u_f(B_s,T-s) \cdot dB_s
\end{align*}
and 
\begin{align*}
S^T_Af(x) &= E(A \ast f|B_0 = (x,0)).
\end{align*}
It is shown in \cite{BanMen} that $\lim_{T\rightarrow \infty} S^T_A = S_A$ exists in $L^2(\mathbb{R}^n)$. Moreover,

\begin{equation}
\label{strong}\|S_Af(x)\|_p \leq (p^*-1)\|A\|\|f\|_p.
\end{equation}

If $A^{(i,j)}$ is defined by
\begin{displaymath}
   a^{(i,j)}_{l,m} = \left\{
     \begin{array}{lr}
       -1 & l=i\:m=j\\
0&\text{otherwise}
     \end{array}
   \right\},
\end{displaymath} 
 then $S_A$ is the second order Riesz transform, $R_iR_j.$ By (\ref{decomp}) this easily leads us to the conclusion that (\ref{strongtype}) holds for the Beurling-Ahlfors transform with constant $2(p^*-1)$. As with the projection operators arising from background radiation, if $A(x,y) = A(y)$ is independent of $x$, then $S_A$ is a Fourier multiplier. Furthermore, we may again find a kernel so that $$S_Af(x) = \int_{\mathbb{R}^n} \calK_A(x,\tilde{x})f(\tilde{x}) d\tilde{x},$$ where 

\begin{equation*}
\label{HeatK} \calK_A(x,\tilde{x}) = \int_0^\infty \int_{\mathbb{R}^n} A(\bar{x},t)\nabla_x h_t(\bar{x}-\tilde{x})\nabla_x h_t(\bar{x}-x) d\bar{x} dt.
\end{equation*}
In light of remark $\ref{densities}$, we see that this is (up to multiplication by a constant) the kernel from (\ref{StableK}) in the case that $\alpha=2$ and $a^{i,j}(x,y)=0$ whenever either $i$ or $j=n+1$. Therefore, the operators considered in theorem \ref{stable} include  the operators from \cite{BanMen} and many of the operators considered in \cite{BanWan}. (In \cite{BanWan} all entries of the matrix, $A(x,y)=(a^{i,j}(x,y)),$ are allowed to depend on both $x$ and $y$, even when $i$ or $j=n+1$. Depending on the choice of $A$, $T_A$ may not satisfy (\ref{smoothness1}) and (\ref{smoothness2}) in that case.)

\section{The Proof of Theorem \ref{stable}}

\begin{proof}
We need to verify that $T_A$ is bounded on $L^2(\mathbb{R}^n)$ and that $K_A$ satisfies the estimates (\ref{size}), (\ref{smoothness1}), and (\ref{smoothness2}). From the definition of $T_A,$ we observe that  (\ref{smoothness1}) and (\ref{smoothness2}) are equivalent.
\begin{lemma} \label{LL2} $T_A$ is bounded on $L^2(\mathbb{R}^n)$. In particular, there exists a constant $C_{n,\alpha}$, depending only on $n$ and $\alpha$, such that for all $f\in C_0^\infty(\mathbb{R}^n)$
\begin{equation}
\| T_A f\|_2 \leq C_{n,\alpha}\|A\|\|f\|_2 .
\end{equation}
\end{lemma}
\begin{proof}
Let $f,g\in C^\infty_0(\mathbb{R}^n)$.We will show that 
\begin{equation*}
\left|\int_{\mathbb{R}^n} T_Af(x)g(x)dx\right|\leq C_{n,\alpha}\|A\|\|f\|_2\|g\|_2.
\end{equation*}
Letting $u_f$ and $u_g$ denote $\varphi_y*f$ and $\varphi_y*g$ respectively,
\begin{align*}
&\left|\int_{\mathbb{R}^n} T_Af(x)g(x)dx\right|\\
 & = \left|\int_{\mathbb{R}^n} \int_{\mathbb{R}^n} K_A(x,\tilde{x})f(\tilde{x})g(x)d\tilde{x}dx\right|\\
&= \left|\int_{\mathbb{R}^n} \int_{\mathbb{R}^n} \int_0^\infty \int_{\mathbb{R}^n} 2y A(\bar{x},y) \nabla \varphi_y(\bar{x}-\tilde{x})\cdot \nabla \varphi_y(\bar{x}-x) f(\tilde{x}) g(x) d\bar{x}dyd\tilde{x}dx\right|\\
&= \left|\int_0^\infty\int_{\mathbb{R}^n}2yA(\bar{x},y)\int_{\mathbb{R}^n}\nabla \varphi_y(\bar{x}-\tilde{x})f(\tilde{x})d\tilde{x}\cdot \int_{\mathbb{R}^n}\nabla \varphi_y(\bar{x}-x)g(x)dxd\bar{x}dy\right|\\
&= \left|\int_0^\infty\int_{\mathbb{R}^n}2yA(\bar{x},y)\nabla u_f(\bar{x},y)\cdot \nabla u_g(\bar{x},y) d\bar{x}dy\right|\\
&\leq 2\|A\|\int_0^\infty \int_{\mathbb{R}^n}y^{1/2}|\nabla u_f(x,y)| y^{1/2}|\nabla u_g(x,y)|dxdy.
\end{align*}
 Now by the Cauchy-Schwartz inequality and Holder's inequality,
\begin{align*}
&\int_0^\infty \int_{\mathbb{R}^n}y^{1/2}|\nabla u_f(x,y)| y^{1/2}|\nabla u_g(x,y)|dxdy\\
&\leq \int_0^\infty \left(\int_{\mathbb{R}^n}y|\nabla u_f(x,y)|^2dx\right)^{1/2} \left(\int_{\mathbb{R}^n}y|\nabla u_g(x,y)|^2dx\right)^{1/2}dy\\
&\leq \left(\int_0^\infty \int_{\mathbb{R}^n}y|\nabla u_f(x,y)|^2dxdy\right)^{1/2}\left(\int_0^\infty \int_{\mathbb{R}^n}y|\nabla u_g(x,y)|^2dxdy\right)^{1/2}.
\end{align*}
The proof will be complete once we show that 
\begin{equation*}
\left(\int_0^\infty y\int_{\mathbb{R}^n}|\nabla u_f(x,y)|^2dxdy\right) \leq C_{n,\alpha}\|f\|^2_2.
\end{equation*}
Since $\varphi$ is the density of $X_1$, which has characteristic function $\mathbb{E}(e^{iX_1\xi})= e^{-|\xi|^\alpha}$, we have that $\widehat{\varphi(\xi)} = e^{-(2\pi|\xi|)^\alpha}.$ Therefore, we may apply Plancherel's theorem, use the scaling relation for the Fourier transform, and substitute $t=y|\xi|$, to see that
\begin{align*}
\int_0^\infty y\int_{\mathbb{R}^n}|\nabla_xu_f(x,y)|^2dxdy &= \int_0^\infty y\int_{\mathbb{R}^n} 4\pi^2|\xi|^2 |\widehat{\varphi_y}(\xi)|^2|\widehat{f}(\xi)|^2 d\xi dy\\
&=C \int_0^\infty y \int_{\mathbb{R}^n} |\xi|^2|\widehat{\varphi}(\xi y)|^2 |\widehat{f}(\xi)|^2d\xi dy\\
&= C \int_0^\infty t \int_{\mathbb{R}^n} |\widehat{\varphi}(\xi' t)|^2|\widehat{f}(\xi)|^2d\xi dt\\
&= C \int_{\mathbb{R}^n}|\widehat{f}(\xi)|^2 \int_0^\infty te^{-2(2\pi t)^\alpha} dt d\xi\\
&\leq C_{n,\alpha} \int_{\mathbb{R}^n}|\widehat{f}(\xi)|^2 d\xi = C_{n,\alpha}\|f\|^2_2.
\end{align*}
Likewise,
\begin{align*}
\int_0^\infty y\int_{\mathbb{R}^n}|\partial_y u_f(x,y)|^2dxdy &= \int_0^\infty y\int_{\mathbb{R}^n}  |\partial_y\widehat{\varphi_y}(\xi)|^2|\widehat{f}(\xi)|^2 d\xi dy\\
&=C \int_0^\infty y \int_{\mathbb{R}^n}  |\partial_y \widehat{\varphi}(\xi y)|^2 |\widehat{f}(\xi)|^2d\xi dy\\
&= C \int_0^\infty y \int_{\mathbb{R}^n} |\xi \cdot \nabla \widehat{\varphi}(\xi y)|^2 |\widehat{f}(\xi)|^2d\xi dy\\
&\leq C \int_0^\infty y \int_{\mathbb{R}^n} |\xi|^2 |\nabla \widehat{\varphi}(\xi y)|^2 |\widehat{f}(\xi)|^2d\xi dy\\
&\leq C \int_{\mathbb{R}^n} |\xi|^4 \int_0^\infty y^3|\widehat{\varphi}(\xi y)|^2 |\widehat{f}(\xi)|^2 dy d\xi\\
&= C\int_{\mathbb{R}^n}|\widehat{f}(\xi)|^2 \int_0^\infty t  |\widehat{\varphi}(\xi' t)|^2dtd\xi \leq C_{n,\alpha}\|f\|^2_2.
\end{align*}
\end{proof}

Now that we know $T_A$ is bounded on $L^2(\mathbb{R}^n)$, we will show that it is, in fact, a CZ operator. It suffices to show that $K^{i,j}_A$ satisfies (\ref{size}) and (\ref{smoothness1}) for $1\leq i,j,\leq n+1$ where
\begin{equation}
 K^{i,j}_A(x,\tilde{x}) = \int_0^\infty \int_{\mathbb{R}^n} 2ya^{i,j}(\bar{x},y)\partial_{x_i} \varphi_y(\bar{x}-\tilde{x})\partial_{x_j} \varphi_y(\bar{x}-x) d\bar{x}dy.
\end{equation}
The following lemma will be used to see that certain integrals converge.

\begin{lemma}\label{Lbounds}
There exists a constant $C_{n,\alpha}$, depending only on $n$ and $\alpha,$ such that for all $x\in \mathbb{R}^n$, $1\leq i,j \leq n$,
\begin{align}
 \label{decay0}|\varphi(x)| &\leq \frac{C_{n,\alpha}}{(1+|x|^2)^{(n+\alpha)/2}}\\
\label{decay1}|\partial_{x_i}\varphi(x)| &\leq \frac{C_{n,\alpha}|x|}{(1+|x|^2)^{(n+2+\alpha)/2}} \leq \frac{C_{n,\alpha}}{(1+|x|^2)^{(n+1+\alpha)/2}}
\end{align}
and
\begin{equation}
\label{decay2}|\partial_{x_i}\partial_{x_j}\varphi(x)| \leq  \frac{C_{n,\alpha}}{(1+|x|^2)^{(n+2+\alpha)/2}}.
\end{equation}
\end{lemma}

\begin{proof}
Inverting the characteristic function of $X_1$ we see
\begin{equation}
\label{invchar}\varphi(x) = \int_{\mathbb{R}^n} e^{-i x \cdot \xi} e^{-|\xi|^\alpha}.
\end{equation}
From this we readily see that $\varphi\in C^\infty(\mathbb{R}^n)$, so in order to show (\ref{decay1}) it suffices to show that there exists a constant $C_{n,\alpha}$ so that 
\begin{equation}
\label{origin}|\partial_{x_i}\varphi(x)|\leq C_{n,\alpha} |x|
\end{equation}
 and 
\begin{equation}
\label{infinity}|\partial_{x_i}\varphi(x)| \leq \frac{C_{n,\alpha}}{|x|^{n+1+\alpha}}.
\end{equation}
Using the fact that 
\begin{equation*}
\int_{\mathbb{R}^n} \xi_i e^{-|\xi|^\alpha} d\xi =0,
\end{equation*}
we see that 
\begin{align*}
|\partial_{x_i}\varphi(x)| &= \left|\int_{\mathbb{R}^n} \xi_i e^{-ix\cdot \xi} e^{-|\xi|^\alpha} d\xi\right|\\
&=   \left|\int_{\mathbb{R}^n} \xi_i (e^{-ix\cdot \xi}-1) e^{-|\xi|^\alpha} d\xi\right|\\
&\leq  \int_{\mathbb{R}^n} |\xi| \:|e^{-ix\cdot \xi}-1| e^{-|\xi|^\alpha} d\xi\\
& \leq 2 \int_{\mathbb{R}^n} |\xi|^2 |x|e^{-|\xi|^\alpha} d\xi \leq C_{n,\alpha} |x|, \\
\end{align*} 
with the last inequality following because 
\begin{equation*}
|e^{ix\cdot \xi}-1| \leq |\cos(x\cdot\xi)-1| + |\sin(x\cdot\xi)| \leq 2 |x\cdot\xi|.
\end{equation*}
Therefore (\ref{origin}) holds. 

To show (\ref{infinity}), we express $X_t$ as a process subordinated to Brownian motion. A subordinator is an a.s. increasing one-dimensional L\'evy process. It is well known (see \cite{Blu} for details) that there exists a subordinator, $T_t$, such that
\begin{equation*}
X_t = B_{T_t},
\end{equation*}
where $B_t$ is a standard Brownian motion (run at twice the usual speed). By conditioning on $T_t$ we see that the density of $X_t$ is given by
\begin{equation*}
\psi_t(x) = \int_0^\infty \frac{1}{(4\pi s)^{n/2}}e^{-|x|^2/4s} \eta^{\alpha/2}(t,s)ds,
\end{equation*}
where $\eta^{\alpha/2}(t,\cdot)$ is the density of $T_t$. Since $\varphi=\psi_1$, we see that 
\begin{align*}
\partial_{x_i} \varphi(x) &= \int_0^\infty \frac{1}{(4\pi s)^{n/2}}\frac{x_i}{s}e^{-|x|^2/4s} \eta^{\alpha/2}\left(1,s\right)ds\\
&=  C_n \frac{x_i}{|x|^n}\int_0^\infty u^{\frac{n}{2}-1}e^{-u} \eta^{\alpha/2}\left(1,\frac{|x|^2}{4u}\right)du.
\end{align*}
It is known (see e.g.  \cite{Bog}) that there exists a constant $C_\alpha$, depending only on $\alpha$, such that 
\begin{equation}
\label{etabound}\eta^{\alpha/2}(t,s)\leq C_\alpha ts^{-1-\alpha/2}.
\end{equation}
 Therefore we have
\begin{equation*}
|\partial_{x_i}\varphi(x)| \leq \frac{C_\alpha}{|x|^{n+1+\alpha}}\int_0^\infty  u^{(n+\alpha)/2}e^{-u} du
\end{equation*}
so (\ref{infinity}) holds. Similar computations show
\begin{equation*}
 |\varphi(x)|\leq \frac{C_\alpha}{|x|^{n+\alpha}} \int_0^\infty u^{(n+\alpha-2)/2} e^{-u}du
\end{equation*}
and
\begin{equation*}
|\partial_{x_i}\partial_{x_j}\varphi(x)|\leq \frac{C_\alpha}{|x|^{n+\alpha+2}} \int_0^\infty u^{(n+\alpha+2)/2}(u+1) e^{-u}du.
\end{equation*}
Moreover, since $\varphi$ is smooth, it and all of its all of its partial derivatives are bounded near the origin. 
Therefore $\varphi$ satisfies (\ref{decay0}) and (\ref{decay2}). 
\end{proof}
We are now poised to prove the theorem.

 \noindent\textbf{Case 1.} Either $i$ or $j=n+1$:

The fact that $a^{(i,j)}(x,y)=a^{(i,j)}(y)$ depends only on $y$ allows us to use the semigroup property of $\psi_y$. Note that \begin{equation*}
\varphi_y * \varphi_y = \psi_{y^\alpha} * \psi_{y^\alpha} = \psi_{2y^\alpha}= \varphi_{2^{1/\alpha}y}.
\end{equation*}
Therefore, substituting $w=\bar{x}-\tilde{x}$ we see that
\begin{align*}
|K^{(i,j)}(x,\tilde{x})| &= \left|\int_0^\infty \int_{\mathbb{R}^n} 2ya^{(i,j)}(y)\partial_{x_i}\varphi_y(w)\partial_{x_j}\varphi_y(w-(x-\tilde{x}))dwdy\right|\\
&= \left|\int_0^\infty 2ya^{(i,j)}(y)\int_{\mathbb{R}^n} \partial_{x_i}\varphi_y(w)\partial_{x_j}\varphi_y(w-(x-\tilde{x}))dwdy\right|\\
&=\left|\int_0^\infty 2ya^{(i,j)}(y)\partial_{x_i}\partial_{x_j}\varphi_{2^{1/\alpha}y}(x-\tilde{x})dy\right|\\
&\leq \|a^{(i,j)}\|_\infty \int_0^\infty 2y |\partial_{x_i}\partial_{x_j}\varphi_{2^{1/\alpha}y}(x-\tilde{x})|dy.
\end{align*}
Likewise,
\begin{align*}
|\partial_{x_k}K^{(i,j)}(x,\tilde{x})| &= \left|\int_0^\infty \int_{\mathbb{R}^n} 2ya^{(i,j)}(y)\partial_{x_i}\varphi_y(w)\partial_{x_k}\partial_{x_j}\varphi_y(w-(x-\tilde{x}))dwdy\right|\\
&\leq \|a^{(i,j)}\|_\infty \int_0^\infty 2y |\partial_{x_i}\partial_{x_j}\partial_{x_k}\varphi_{2^{1/\alpha}y}(x-\tilde{x})|dy.
\end{align*}
Therefore, it suffices to show that there exists a constant $C_{n,\alpha}$ so that  $|K(x)|\leq C_{n,\alpha}\frac{1}{|x|^n}$ and $|K'(x)| \leq C_{n,\alpha}\frac{1}{|x|^{n+1}}$ for all $x\neq 0$, where
\begin{equation*}
K(x) =  \int_0^\infty 2y |\partial_{x_i}\partial_{x_j}\varphi_{2^{1/\alpha}y}(x)|dy
\end{equation*}
and
\begin{equation*}
K'(x) =  \int_0^\infty 2y |\partial_{x_i}\partial_{x_j}\partial_{x_k}\varphi_{2^{1/\alpha}y}(x)|dy.
\end{equation*}
$\varphi_y$ is homogeneous of order $-n$, so its $k-th$ order partial derivatives are homogeneous of order $-n-k$. Therefore, if we make the substitution $y=|x|t$ we have
\begin{align*}
K(x) &=  \int_0^\infty 2y |\partial_{x_i}\partial_{x_j}\varphi_{2^{1/\alpha}y}(x)|dy\\
&= \int_0^\infty 2|x|t|\partial_{x_i}\partial_{x_j}\varphi_{2^{1/\alpha}|x|t}(|x|x')||x|dt\\
&= \int_0^\infty 2|x|t\frac{1}{|x|^{n+2}}|\partial_{x_i}\partial_{x_j}\varphi_{2^{1/\alpha}t}(x')||x|dt\\
&= \frac{1}{|x|^{n}}\int_0^\infty 2t|\partial_{x_i}\partial_{x_j}\varphi_{2^{1/\alpha}t}(x')|dt
\end{align*}
where $x'=\frac{x}{|x|}$. Similarly,
\begin{equation*}
K'(x) = \frac{1}{|x|^{n+1}}\int_0^\infty 2t|\partial_{x_i}\partial_{x_j}\partial_{x_k}\varphi_{2^{1/\alpha}t}(x')|dt.
\end{equation*}
The lemma will be proved as soon as we bound the above integrals. We have assumed that either $i$ or $j=n+1,$ so  we need to bound the following four integrals for any $1\leq k,l \leq n$. 
\begin{align*}
&\int_0^\infty 2t|\partial_t\partial_t\varphi_{2^{1/\alpha}t}(x')|dt,
&\int_0^\infty 2t|\partial_t\partial_{x_k}\varphi_{2^{1/\alpha}t}(x')|dt,\\
&\int_0^\infty 2t|\partial_t\partial_{x_k}\partial_{x_l}\varphi_{2^{1/\alpha}t}(x')|dt,
&\int_0^\infty 2t|\partial_{t}\partial_{t}\partial_{x_k}\varphi_{2^{1/\alpha}t}(x')|dt.
\end{align*}
We will show how to bound the first integral. The other three may be bounded by the exact same method. Recalling that $\varphi_t(x)=\frac{1}{t^n}\varphi\left(\frac{x}{t}\right)$, we see that 
\begin{equation*}
\partial_t\partial_t\varphi_{2^{1/\alpha}t}(x) =  \frac{C^{(1)}_n}{t^{n+2}} \varphi\left(\frac{x}{t}\right) + \frac{C^{(2)}_n}{t^{n+3}}\sum_{i=1}^n x_i\partial_{x_i}\varphi\left(\frac{x}{t}\right) + \frac{C^{(3)}_n}{t^{n+4}}\sum_{i=1}^n\sum_{j=1}^n x_ix_j\partial_{x_i}\partial_{x_j}\varphi\left(\frac{x}{t}\right),
\end{equation*} 
where $C^{(1)}_n, C^{(2)}_n$ and $C^{(3)}_n$ are constants depending on $n.$

Therefore, it suffices to bound 
\begin{equation*}
\int_0^\infty \frac{t}{t^{n+a}}\partial^{\beta}\varphi\left(\frac{x'}{t}\right)dt
\end{equation*}
when $a=2,3,$ or $4$ and $\beta$ is a multi-index with $|\beta|=a-2$. By (\ref{decay0}), (\ref{decay1}), and (\ref{decay2}), we have 
\begin{equation*}
\partial^{\beta}\varphi(x) \leq \frac{C_{n,\alpha}}{(1+|x|^2)^{(n+\alpha + |\beta|)/2}},
\end{equation*}
which implies
\begin{equation*}
\partial^{\beta}\varphi\left(\frac{x}{t}\right) \leq \frac{C_{n,\alpha}t^{n+\alpha + |\beta|}}{(t^2+|x|^2)^{(n+\alpha + |\beta|)/2}}.
\end{equation*}
Therefore,
\begin{align*}
\int_0^\infty \frac{t}{t^{n+a}}\partial^{\beta}\varphi\left(\frac{x'}{t}\right)dt &\leq  \int_0^\infty\frac{t}{t^{n+a}} \frac{t^{n+\alpha + |\beta|}}{(t^2+1)^{(n+\alpha + |\beta|)/2}}dt\\
&= \int_0^\infty \frac{t^{\alpha-1}}{(1+t^2)^{(n+\alpha + |\beta|)/2}}dt < \infty.
\end{align*}

\noindent\textbf{Case 2.} $1\leq i,j \leq n$:

Since $a^{i,j}(x,y)$ depends on both $x$ and $y,$ we are unable to use the semi-group property of $\psi_y$. We are, however, still able to pull out $\|A\|$ and use homogeneity. This again allows us to bound our kernel by the product of  $\frac{1}{|x-\tilde{x}|^n}$ and an integral. As in case 1, we start out by substituting $w=x-\tilde{x}$ to see
\begin{align*}
 |K^{(i,j)}(x,\tilde{x})| &\leq \|A\|\int_0^\infty\int_{\mathbb{R}^n} 2y|\partial_{x_i}\varphi_y(w)||\partial_{x_j}\varphi_y(w-(\tilde{x}-x))|dwdy \text{ and}\\
|\partial_{x_k}K^{(i,j)}(x,\tilde{x})| &\leq \|A\|\int_0^\infty\int_{\mathbb{R}^n} 2y|\partial_{x_i}\varphi_y(w)||\partial_{x_k}\partial_{x_j}\varphi_y(w-(\tilde{x}-x))|dwdy.
\end{align*}
Therefore, we need to show that there exists a constant $C_{n,\alpha}$ so that $|K(x)| \leq C_{n,\alpha}\frac{1}{|x|^n}$ and $|K'(x)| \leq C_{n,\alpha} \frac{1}{|x|^{n+1}}$ for all $x\neq 0$ where now
\begin{equation*}
K(x) = \int_0^\infty\int_{\mathbb{R}^n} 2y|\partial_{x_i}\varphi_y(w)||\partial_{x_j}\varphi_y(w-x)|dwdy
\end{equation*}
and
\begin{equation*}
K'(x) = \int_0^\infty\int_{\mathbb{R}^n} 2y|\partial_{x_i}\varphi_y(w)||\partial_{x_k}\partial_{x_j}\varphi_y(w-x)|dwdy.
\end{equation*}
Using homogeneity and substituting $y=t|x|$ and $w = |x|z$ we see that 
\begin{align}
|K(x)| &=\int_0^\infty\int_{\mathbb{R}^n} 2y|\partial_{x_i}\varphi_y(w)||\partial_{x_j}\varphi_y(w-x)|dwdy \nonumber\\
 &=\int_0^\infty\int_{\mathbb{R}^n} 2t|x||\partial_{x_i}\varphi_{|x|t}(|x|z)||\partial_{x_j}\varphi_{|x|t}(|x|(z-x'))||x|^{n+1}dzdt \nonumber\\
\label{Int1}&= \frac{1}{|x|^n}  \int_0^\infty\int_{\mathbb{R}^n} 2t|\partial_{x_i}\varphi_{t}(z)||\partial_{x_j}\varphi_t(z-x')|dzdt.
\end{align}
Similarly, we have 
\begin{equation}
\label{Int2}|K'(x)| = \frac{1}{|x|^{n+1}}  \int_0^\infty\int_{\mathbb{R}^n} 2t|\partial_{x_i}\varphi_{t}(z)||\partial_{x_k}\partial_{x_j}\varphi_t(z-x')|dzdt.
\end{equation}
Therefore, to complete the proof, we need to show that the integrals in (\ref{Int1}) and (\ref{Int2}) are convergent. (Note that a simple rotation of coordinates shows they do not depend on $x'$.) We will show that the integral in (\ref{Int2}) converges. The proof that the integral in ($\ref{Int1}$) converges is similar.

 By (\ref{decay1}) and (\ref{decay2}) we know that there exists a constant $C_{n,\alpha}$ so
\begin{equation}
|\partial_{x_i} \varphi(x)| \leq \frac{C_{n,\alpha}|x|}{(1+|x|^2)^{(n+2+\beta)/2}}
\end{equation}
and 
\begin{equation}
|\partial_{x_i} \partial_{x_j}\varphi(x)| \leq \frac{C_{n,\alpha}}{(1+|x|^2)^{(n+2+\beta)/2}},
\end{equation}
where $\beta = \min\{\alpha,\frac{1}{2}\}$. (The fact that $\beta \leq \frac{1}{2}$ will be used to see that a certain integral is convergent.)
Therefore,
\begin{equation}
\label{firstpartial}|\partial_{x_i} \varphi_t(x)| \leq \frac{C_nt^\beta |x|}{(t^2+|x|^2)^{(n+2+\beta)/2}}
\end{equation}
and 
\begin{equation}
\label{secondpartial} |\partial_{x_i} \partial_{x_j}\varphi_t(x)| \leq \frac{C_nt^\beta}{(t^2+|x|^2)^{(n+2+\beta)/2}}.
\end{equation}
This allows us to  see that 
\begin{align*}
&\int_0^\infty\int_{\mathbb{R}^n} 2t|\partial_{x_j}\partial_{x_i}\varphi_{t}(z)||\partial_{x_j}\varphi_t(z-x')|dzdt \\
&\leq \int_{\mathbb{R}^n} \int_0^\infty t\frac{t^\beta}{(|z|^2+t^2)^{(n+2+\beta)/2}}\frac{|x'-z|t^\beta}{(|x'-z|^2+t^2)^{(n+2+\beta)/2}}  dt dz,
\end{align*}
so it suffices to show that $g_1(z)$ and $g_2(z)$ are integrable over $\mathbb{R}^n$ for
\begin{align*}
g_1(z) &= \int_0^1 t^{1+2\beta}\frac{1}{(|z|^2+t^2)^{(n+2+\beta)/2}}\frac{|x'-z|}{(|x'-z|^2+t^2)^{(n+2+\beta)/2}}  dt\quad\text{and}\\
 g_2(z) &=  \int_1^\infty t^{1+2\beta}\frac{1}{(|z|^2+t^2)^{(n+2+\beta)/2}}\frac{|x'-z|}{(|x'-z|^2+t^2)^{(n+2+\beta)/2}} dt.
\end{align*}
 
If $z_n$ is a sequence converging to $z$, then
\begin{align*}
\lim_{n\rightarrow \infty} g(z_n) &= \lim_{n\rightarrow \infty} \int_1^\infty t^{1+2\beta}\frac{1}{(|z_n|^2+t^2)^{(n+\beta+2)/2}}\frac{|x'-z_n|}{(|x'-z_n|^2+t^2)^{(n+\beta+2)/2}} dt\\
&=  \int_1^\infty t^{1+2\beta}\frac{1}{(|z|^2+t^2)^{(n+2+\beta)/2}}\frac{|x'-z|}{(|x'-z|^2+t^2)^{(n+2+\beta)/2}} dt=g_2(z),
\end{align*}
with the middle inequality justified by the dominating convergence theorem applied to $|x'-z|\frac{1}{t^{2n+3}}$. Thus $g_2(z)$ is continuous on $\mathbb{R}^n$.
Furthermore, for large $z$, substituting $t=|z|\tan(\theta)$ allows us to see that
\begin{align*}
g_2(z) &\leq C_{n,\beta}\int_1^\infty t^{1+2\beta} \frac{|z|}{(|z|^2+t^2)^{n+\beta+2}}dt\\
&\leq C_{n,\beta}\int_0^{\pi/2}\frac{|z|^{1+2\beta}\tan^{1+2\beta}(\theta)|z| |z|\sec^2(\theta)}{|z|^{2n+4+2\beta}\sec^{2n+4+2\beta}(\theta)} d\theta\\
&= C_{n,\beta}\frac{1}{|z|^{2n+1}} \int_0^{\pi/2} \sin^{1+2\beta}(\theta)\cos^{2n+1}(\theta)d\theta,
\end{align*}
so $g_2(z)$ is integrable.

Likewise, we can see that $g_1(z)$ is continuous on $\mathbb{R}^n \setminus \{0,x'\}$ using by applying the dominating convergence theorem with $t^{2\beta+1}|z|^{-n-\beta-2}|x'-z|^{-n-\beta-1}$, and for large $z$ we have
\begin{equation*}
g_1(z) \leq C_{n,\beta}\frac{1}{|z|^{2n+1}} \int_0^{\pi/2} \sin^{1+2\beta}(\theta)\cos^{2n+1}(\theta)d\theta.
\end{equation*}
Therefore, it remains to show that $g_1(z)$ is integrable near $0$ and $x'$. 

If $|z|<1/2$ and $0<t<1$, it is easy to see
\begin{equation*}
\frac{|x'-z|}{(|x'-z|^2+t^2)^{(n+2+\beta)/2}} \leq C_{n,\beta}, 
\end{equation*}
so again substituting $t=|z|\tan(\theta)$ we see
\begin{align*}
g_1(z) &\leq C_{n,\beta} \int_0^1 \frac{t^{1+2\beta}}{(|z|^2+t^2)^{(n+2+\beta)/2}}dt\\
&\leq C_{n,\beta} \frac{1}{|z|^{n-\beta}}\int_0^{\pi/2} \sin^{2\beta}(\theta)\cos^{n-\beta-1}(\theta)d\theta.
\end{align*}
Since $\beta \leq \frac{1}{2}$, the last integral is finite, so $g_1(z)$ is integrable near 0. A simple change of variables and a nearly identical computation shows that $g_1(z)$ is integrable near $x'$, so therefore $g_1(z)$ is integrable on all of $\mathbb{R}^n$ which completes the proof. 

\end{proof}
We end this section by remarking that if $i$ or $j=n+1$, then the integral in (\ref{Int2}) is divergent. This is why we need the assumption that $a^{i,j}(x,y)=a^{i,j}(y)$ in that case. 

\section{Closing Remarks}
Examining the proof of theorem \ref{stable}, we see that the only facts we used were the homogeneity of $\varphi_y(x)$, the fact that $\widehat{\varphi}$ is ``small enough'' to cause $T_A$ to be bounded on $L^2(\mathbb{R}^n)$, and the bounds (\ref{decay0}), (\ref{decay1}), and (\ref{decay2}). This immediately gives us the following corollary.

\begin{corollary} \label{cond} Let $\phi\in C^2(\mathbb{R}^n)$ satisfy (\ref{decay0}), (\ref{decay1}), and (\ref{decay2}), and for $y>0,$ let $\phi_y= \frac{1}{y^n}\phi\left(\frac{x}{y}\right).$ Assume that there exists a constant $C$ such that for all $\xi' \in S^{n-1}$
\begin{equation}
\label{fourierray} \int_0^\infty t\widehat{\phi}(t\xi')^2dt < C.
\end{equation}
Let $A(x,y)$ be as in theorem \ref{stable}.
 Consider the kernel 
\begin{equation*}
 K_A(x,\tilde{x}) = \int_0^\infty \int_{\mathbb{R}^n} 2yA(\bar{x},y)\nabla \phi_y(\bar{x}-\tilde{x})\nabla \phi_y(\bar{x}-x) d\bar{x}dy,
\end{equation*}
where $\nabla = (\partial_{x_1},\ldots,\partial_{x_n},\partial_y)$.
 Then the operator 
\begin{equation*}
T_{A}f(x) = \int_{\mathbb{R}^n} K(x,\tilde{x}) f(\tilde{x})d\tilde{x}\
\end{equation*}
 is a CZ operator.
\end{corollary}

The key to proving lemma (\ref{Lbounds}) was the fact that we could write the $\alpha$-stable process as $B_{T_t}$ where $T_t$ is the $\alpha/2$ stable subordinator and $B_t$ is an independent Brownian motion (run at twice the usual speed). This motivates the following question. Let $T_t$ be a subordinator, let $B_t$ be an independent Brownian motion, and let $X_t=B_{T_t}$. Under what conditions on $T_t$ does the density of $X_1$ satisfy the conditions of corollary \ref{cond}?

If $X_t=B_{T_t}$ is any such process, called subordinate Brownian motion in the literature, it is well known (see for example \cite{Kim}) that there exists a function $\Phi:[0,\infty)\rightarrow [0,\infty)$, called the Laplace exponent of $T_t$, such that
\begin{equation}
\label{Laplace}e^{-\lambda T_t} = e^{-t\Phi(\lambda)}
\end{equation}
and that the L\'evy symbol of $X_t$ is given by
\begin{equation*}
\rho(\xi) = -\Phi\left(|\xi|^2\right).
\end{equation*}
Inspecting the proofs of lemma \ref{LL2} and lemma \ref{Lbounds}, we see that in order for the density of $X_1$ to satisfy the conditions of corollary \ref{cond}, it suffices to have a bound similar to (\ref{etabound}) on the density of $T_1$, and for $\Phi(\lambda)$ to increase fast enough as $\lambda \rightarrow \infty$ for the integrals in the proofs to converge. We summarize this in the following corollary.

\begin{corollary} \label{subcor} Let $X_t=B_{T_t}$ where $T_t$ is a subordinator and $B_t$ is an independent Brownian motion run at twice the usual speed.  Let $\phi$ denote the density of $X_1$, and for $y>0$, let $\phi_y(x)=\frac{1}{y^n}\phi\left(\frac{x}{y}\right)$. Let $\Phi$ be the Laplace exponent of $T_t$ and assume that there exists some $\delta >0$ so that
\begin{equation}
\Phi(\lambda) \geq O(\lambda^\delta),\,\,\, \: \text{as} \: \lambda \rightarrow \infty.
\end{equation}
 Further assume that there exist a constants $C$ and  $\gamma>0$ such that the density of $T_1$, $\eta(1,\cdot)$, satisfies
\begin{equation}
\eta(1,s) \leq C s^{-1-\gamma/2}
\end{equation}
for all $s>0$. Let $A(x,y)$ be as in theorem \ref{stable} and consider the kernel 
\begin{equation*}
 K_A(x,\tilde{x}) = \int_0^\infty \int_{\mathbb{R}^n} 2yA(\bar{x},y)\nabla \phi_y(\bar{x}-\tilde{x})\nabla \phi_y(\bar{x}-x) d\bar{x}dy,
\end{equation*}
where $\nabla = (\partial_{x_1},\ldots,\partial_{x_n},\partial_y)$.
 Then the operator 
\begin{equation*}
T_{A}f(x) = \int_{\mathbb{R}^n} K(x,\tilde{x}) f(\tilde{x})d\tilde{x}\
\end{equation*}
 is a CZ operator.
\end{corollary}

An interesting example of subordinate Brownian motion is provided by the so called relativistic $\alpha$-stable processes. The relativistic $\alpha$-stable process with mass $m\geq 0$ is the L\'evy process, $X^m_t$, with L\'evy symbol 
\begin{equation}
 \label{relchar}\rho(\xi) = -\left((|\xi|^2+m^{2/\alpha})^{\alpha/2}-m\right).
\end{equation}
These processes have been widely studied in recent years because of their importance in relativistic quantum mechanics. (See \cite{Chen} and the references provided there for more information.) In \cite{Ryznar} it is shown that $T_t$, the subordinator for $X^m_t,$ has density
\begin{equation}
\label{relsub} \eta^{m,\alpha/2}(t,s) = e^{mt}e^{-m^{2/\alpha}s}\eta^{\alpha/2}(t,s),
\end{equation}
and Laplace exponent 
\begin{equation}
\label{rellap} \Phi(\lambda) = (\lambda+m^{2/\alpha})^{\alpha/2}-m.
\end{equation}
Therefore, we readily see that the conditions of corollary \ref{subcor} are satisfied.

The motivation of this paper was to answer questions left open in \cite{BanMen} and \cite{BanWan}. Are the operators considered in those papers weak-type $(1,1)$ in addition to being strong-type $(p,p)$ for $1<p<\infty$? Proving that these operators are CZ shows that the answer to this question is, in fact, yes. However, CZ operators are also known to satisfy a number   of other desirable properties. For example, they boundedly map the Hardy space $H^1(\mathbb{R}^n)$ to $L^1(\mathbb{R}^n)$ and the space of functions with bounded mean oscillation (BMO) to $L^\infty(\mathbb{R}^n)$. More precisely, if $T$ is any CZ operator, then there exist universal constants $C_n$ and $C'_n$, which depend only on $n,$ so that 
\begin{equation*}
\|T\|_{H^1\rightarrow L^1} \leq C_n (\kappa + \|T\|_{L^2\rightarrow L^2})
\end{equation*}
and
\begin{equation*}
\|T\|_{BMO \rightarrow L^\infty} \leq C'_n (\kappa + \|T\|_{L^2\rightarrow L^2})
\end{equation*}
where $\kappa$ is as in (\ref{size}), (\ref{smoothness1}) and (\ref{smoothness2}). For details on this topic, see \cite[ch. 8]{Gra}.

Another interesting  property of CZ operators is that they are bounded on certain weighted $L^p$ spaces. A weight is a function $w\in L^1_{\text{loc}}(\mathbb{R}^n)$ which is positive almost everywhere. The associated space $L^p(w)$, $1\leq p <\infty$, is the collection of functions $f$ on $\mathbb{R}^n$ such that 
\begin{equation*} 
\|f\|^p_{L^p(w)} = \int_{\mathbb{R}^m} |f(x)|^p w(x) dx < \infty.
\end{equation*}
The Muckenhoupt characteristic of $w$ is defined as 
\begin{equation}
\|w\|_{A_p} = \sup_Q \frac{1}{|Q|} \int_Q w dx \cdot \left(\frac{1}{|Q|}\int_Q w^{-1/(p-1)} dx \right)^{p-1},
\end{equation}
with the supremum taken over all cubes, $Q$.
Note that when $p=2$ this becomes
\begin{equation*}
\|w\|_{A_2} = \sup_Q \frac{1}{|Q|} \int_Q w dx \cdot \frac{1}{|Q|} \int_Q w^{-1} dx.
\end{equation*}
$w$ is said to be an $A_p$ weight if $\|w\|_{A_p}$ is finite. In this case, it is well known (see for example \cite[ch. 9]{Gra}) that if $T$ is a CZ operator, then there exists a constant $C_{n,p,T,w}$, depending on the $n$, $p$, $T$, and $w$, such that 
\begin{equation*}
\|Tf\|_{L^p(w)} \leq C_{n,p,T,w} \|f\|_{L^p(w)},
\end{equation*}
for all $f\in L^p(w)$ when $1<p<\infty$. (A corresponding weak-type result holds when $p=1$.)

Recently, in \cite{Hyt}, Hyt\"onen proved the so called ``$A_2$ conjecture,'' that $C_{n,2,T,w}$ depends linearly on $\|w\|_2$, i.e., there exists a constant $C_{n,2,T}$ such that 
\begin{equation*}
\|Tf\|_{L^2(w)} \leq C_{n,2,T} \|w\|_{A_2} \|f\|_{L^2(w)},
\end{equation*} 
for all $f\in L^2(w)$. Combining this with a result of Dragi\v{c}evi\'c, Grafakos, Pereyra, and Petermichl \cite{Dra} shows
that 
\begin{equation*}
\|Tf\|_{L^p(w)} \leq C_{n,p,T} \|w\|_{A_p}^{\max\{1,\frac{1}{p-1}\}}\|f\|_{L^p(w)}
\end{equation*}
for all $f\in L^p(w).$
For more information weighted $L^p$ spaces and the $A_2$ conjecture, see \cite[ch. 9]{Gra} and \cite{Hyt}. 

The operators considered in \cite{BanMen} are generalized in \cite{AppBan} and \cite{BanBog} by taking the projections of martingales transforms involving more general L\'evy processes in place of Brownian motion. These more general operators are shown in these papers to obey the same $``p^*-1"$ strong-type bound for $1<p<\infty$ as the operators from \cite{BanMen}. In the current paper we have shown that the operators considered in \cite{BanMen} are CZ operators, and therefore are also weak-type (1,1).  It would be interesting to know if the same is true of the operators studied in  \cite{AppBan} and \cite{BanBog}.
\bigskip

\noindent\textbf{Acknowledgments.} I would like to thank my advisor, Rodrigo Ba\~nuelos, for his invaluable help and encouragement while writing this paper.

\end{document}